\documentclass[pdflatex,sn-mathphys-num]{sn-jnl}

\usepackage{graphicx}
\usepackage{amsmath,amssymb,amsfonts,amsthm}
\usepackage{algorithm,algorithmic}
\usepackage{multirow,booktabs}
\usepackage{xcolor,microtype,placeins,needspace}
\usepackage{aliascnt}

\theoremstyle{thmstyleone}
\newtheorem{theorem}{Theorem}[section]
\newtheorem{lemma}{Lemma}[section]
\newtheorem{proposition}{Proposition}[section]

\newaliascnt{assumption}{theorem}
\newtheorem{assumption}[assumption]{Assumption}
\aliascntresetthe{assumption}
\theoremstyle{thmstyletwo}

\theoremstyle{thmstylethree}
\newtheorem{definition}{Definition}[section]

\usepackage[capitalise,noabbrev]{cleveref}
\crefname{assumption}{Assumption}{Assumptions}
\Crefname{assumption}{Assumption}{Assumptions}
\hypersetup{hypertexnames=false,colorlinks=true,linkcolor=blue,citecolor=blue,urlcolor=blue,
  bookmarksopen=true,bookmarksnumbered=true}

\begin{document}
\title[Smoothing Acceleration with Backtracking]{A Smoothing Accelerated Proximal Gradient Method with Nonmonotone Backtracking Strategy for Convex Nonsmooth Composite Multiobjective Optimization}

\author[1]{\fnm{Chengzhi} \sur{Huang}}\email{huangczmaths@163.com}

\affil[1]{\orgname{Sichuan University}, \orgaddress{\city{Chengdu}, \postcode{610065}, \state{Sichuan}, \country{China}}}

\abstract{
This paper studies convex nonsmooth composite multiobjective optimization problems, 
whose objectives combine a convex term admitting a smooth approximation with 
a proper, closed, and convex term. 
A natural way is to smooth the nonsmooth term and apply proximal gradient method, but reducing the smoothing parameter shrinks step sizes, and solving successive smooth approximations to approach the original solution is computationally expensive.
To address this issues, we propose a smoothing 
accelerated proximal gradient method with nonmonotone backtracking strategy. 
Each outer iteration starts from a reduced base smoothness estimate, which is 
increased only if the criterion fails, while the momentum update is 
linked to the estimate through an exact identity used in 
the convergence analysis. Under a compact level-set assumption and bounds of parameters, we use merit function to establish rates of
$\mathcal{O}(k^{-\sigma})$, $\mathcal{O}((\log k)/k)$, and 
$\mathcal{O}(k^{\sigma-2})$ for different smoothing parameters.
Then quasi-Fej\'er monotonicity yields convergence of the full sequence to a weak Pareto solution. 
Numerical experiments show that the method 
can reduce computational cost while preserving high-quality Pareto fronts and examine how the smoothing parameter affects performance.
}


\maketitle

\section{Introduction} \label{sec:intro}

\begingroup
We consider the convex nonsmooth composite multiobjective optimization
problem (NCMOP)
\begin{equation}\label{eq:mop}
    \min_{x \in \mathbb{R}^n}
    F(x)
    =
    \bigl(f_1(x)+g_1(x),\ldots,f_m(x)+g_m(x)\bigr)^T,
\end{equation}
where each $f_i:\mathbb{R}^n\to\mathbb{R}$ is convex, possibly
nonsmooth, and admits a smooth approximation, while each
$g_i:\mathbb{R}^n\to\mathbb{R}\cup\{+\infty\}$ is proper, closed,
and convex. 
The decomposition reflects differences in the
computational treatment of the two terms: the proximal maps of
the $f_i$ need not be readily computable, whereas the structure
of the $g_i$ can be retained in proximal subproblems.
For example, the objective
$\|Ax-b\|_1+\lambda\|x\|_1$
\cite[Example 2.1]{wu2023smoothing}
combines a nonsmooth data-fitting term with an $\ell_1$ regularizer.
For a general matrix $A$, the former need not admit a closed-form
proximal map, while the latter is handled by soft thresholding
\cite{combettes2005signal}.
Multiple observation pairs $(A_i,b_i)$ can be modeled by a vector
of such objectives.

Two widely studied approaches to multiobjective optimization are
evolutionary algorithms and first-order methods
\cite{lapucci2023memetic}.
Representative evolutionary algorithms include the dominance-based
NSGA-II \cite{deb2002fast} and the decomposition-based MOEA/D
\cite{zhang2007moea}.
By evolving a population of candidate solutions, these methods
seek diverse approximations of the Pareto front without requiring
derivative information.
Their ability to explore different regions of the search space
is particularly useful for nonconvex and black-box problems.
However, population evolution can require many objective
evaluations, and function-value-based search does not directly
exploit available gradient information or proximal structure
\cite{lapucci2023memetic}.

First-order methods instead use local analytic information to
construct search directions.
For differentiable objectives, the component gradients can be
combined to obtain a common descent direction at a point that is not Pareto stationary.
The extension to nonsmooth objectives is less direct: replacing
each gradient by an arbitrary subgradient need not produce
descent for all objectives \cite{gebken2021efficient}.
Methods for this setting therefore use richer local information,
including Goldstein-subdifferential approximations
\cite{gebken2021efficient}, proximal bundle models
\cite{hoseini2022bundle}, and subgradient methods with refined
line searches \cite{maleknia2024subgradient}.
Related descent constructions have also been extended to Hilbert
spaces \cite{sonntag2024hilbert}.

These approaches accommodate general nonsmooth objectives, but
computing an acceptable step can involve additional inner work.
Subdifferential approximations may require repeated subgradient
evaluations and direction refinements, while bundle methods
maintain and update cutting-plane models and may perform null
steps before accepting a new iterate.
For \eqref{eq:mop}, there is a further consideration:
modeling each full objective $f_i+g_i$ solely through subgradient
information does not directly exploit the separate structure
of $g_i$.
When this structure permits efficient proximal treatment,
retaining $g_i$ in the subproblem avoids reconstructing its
behavior from subgradient information alone.

This distinction motivates proximal gradient (PG) methods.
In the standard smooth-plus-nonsmooth setting, a PG step
linearizes the smooth term while retaining the convex nonsmooth
term in a quadratically regularized subproblem
\cite{combettes2005signal}.
When the resulting subproblem is tractable, this splitting
exploits the two terms through different computational operations
rather than treating the full objective as a generic nonsmooth
function.
Tanabe et al.~\cite{tanabe2019proximal} extended this principle
to multiobjective optimization by combining the componentwise
models in a common proximal subproblem.
For NCMOP, where the $f_i$ may also be nonsmooth, this framework
suggests smoothing these terms while preserving the $g_i$ in
proximal form.

For smooth $f_i$, the multiobjective PG subproblem takes the form
\endgroup
\[
\begin{aligned}
P_\ell(x,y):=\mathop{\rm argmin}_{z\in\mathbb R^n}\biggl\{
&\max_{1\le i\le m}\bigl[
\langle\nabla f_i(y),z-y\rangle+g_i(z)+f_i(y)-F_i(x)
\bigr]\\[-2pt]
&+\frac\ell2\|z-y\|^2\biggr\}.
\end{aligned}
\]
The PG iteration takes $x_{k+1}=P_\ell(x_k,x_k)$. The min--max subproblem seeks a direction shared by the objectives without prescribing scalarization weights in advance, while the convex terms $g_i$ remain in the subproblem. This uses the composite structure directly and provides Pareto-stationarity guarantees \cite{tanabe2019proximal}. In the convex setting, the associated merit function has rate $\mathcal O(1/k)$ \cite{tanabe2023convergence}. The accelerated proximal gradient (APG) method of \cite{tanabe2023accelerated} employs the extrapolated iteration
$$
y_k=x_k+\frac{t_k-1}{t_{k+1}}(x_k-x_{k-1}),\qquad
x_{k+1}=P_\ell(x_k,y_k),\qquad
t_{k+1}=\frac{1+\sqrt{1+4t_k^2}}{2},
$$
and attains an $\mathcal O(1/k^2)$ convergence rate under convexity, assuming that a fixed parameter $\ell$ exceeding the Lipschitz constant is known or can be chosen in advance. For the scalar composite problem, the same Nesterov-type extrapolation underlies the fast iterative shrinkage-thresholding algorithm (FISTA) \cite{beck2009fast}, which attains the corresponding $\mathcal O(1/k^2)$ rate.
Both formulations require differentiable forward terms, so they do not directly cover a nonsmooth $f_i$.


A natural approach is to replace each nonsmooth forward term \(f_i\) by a smooth approximation \(\widetilde f_i(\cdot,\mu)\), while retaining \(g_i\) in the proximal subproblem. This makes proximal gradient iterations applicable without requiring an efficiently computable proximal map for \(f_i\). Smoothing accelerated proximal gradient methods have been developed for scalar optimization \cite{wu2023smoothing}. 

As $\mu$ tends to zero, the approximation becomes more accurate, but the gradient Lipschitz bound typically grows as $\mathcal O(\mu^{-1})$ \cite{nesterov2005smooth,chen2012smoothing}, restricting admissible step sizes. A global curvature bound may
also substantially exceed the curvature encountered along the iterates and lead to unnecessarily conservative steps. Inaddition,approaching a solution of the original problem by solving progressively refined smooth approximation problem scan require substantial computational work. These difficulties motivate combining backtracking,which adapts the curvature estimate to the current smoothing model, with acceleration aimed at improving progress across the sequence of smoothed models.

Motivated by these considerations, we propose the Smoothing Accelerated Proximal Gradient Method with Nonmonotone Backtracking Strategy (SAPGM-Backtracking). To the best of our knowledge, this is the first smoothing accelerated proximal gradient method for convex nonsmooth composite multiobjective optimization that combines nonmonotone backtracking with a momentum update coupled to the effective curvature of each trial. The effective curvature $\ell_k=L_k/\mu_k$ incorporates both the base smoothness estimate and the smoothing parameter, allowing the extrapolation to respond to changes in both quantities. With the smoothed mapping $p_\ell$ defined in Section \ref{sec:problem}, the accepted iteration is given below.
\[
\begin{aligned}
\mu_{k+1}&=\frac{\mu_0}{(k+2)^\sigma},\qquad
\ell_{k+1}=\frac{L_{k+1}}{\mu_{k+1}},\qquad
t_{k+1}=\frac{1+\sqrt{1+4(\ell_{k+1}/\ell_k)t_k^2}}2,\\
y_k&=x_k+\frac{t_k-1}{t_{k+1}}(x_k-x_{k-1}),\qquad
x_{k+1}=p_{\ell_{k+1}}(x_k,y_k,\mu_{k+1}).
\end{aligned}
\]

Each backtracking search starts from the reduced base estimate $L_k/\eta$, allowing the method to test a larger step at the current smoothing level. The trial estimate is increased geometrically only when the componentwise majorization condition fails, following the nonmonotone-backtracking strategy of [16]. This mechanism allows the estimate to adapt to the current model instead of retaining a potentially conservative value from earlier iterations.

For every trial, the momentum parameter, the extrapolation point, and the proximal candidate are recomputed using the corresponding effective curvature. The extrapolation therefore adjusts to changes in both the smoothing parameter and the trial step size throughout the search. Here, nonmonotonicity refers to the accepted base smoothness estimates, which may increase or decrease between outer iterations; the trial estimates within each inner search increase until the majorization condition is satisfied.

Our contributions are as follows.
\begin{enumerate}
\item We develop a smoothing accelerated proximal gradient method that adapts both the step size and the extrapolation coefficient to each successive smoothing model. Its nonmonotone backtracking strategy allows the baseline smoothness estimate to increase or decrease, while the momentum is updated according to the effective curvature at each trial point. To the best of our knowledge, this is the first smoothing accelerated proximal gradient method for convex nonsmooth composite multiobjective optimization that combines nonmonotone backtracking with momentum updates coupled to the effective curvature at each trial.

\item We establish convergence rates for a merit function of the original nonsmooth problem. Under a compact level-set assumption and uniform positive lower and upper bounds on the base smoothness estimates, the schedule $\mu_k=\mu_0/(k+1)^\sigma$ yields rates of $\mathcal O(k^{-\sigma}), \mathcal O((\log k)/k)$, and $\mathcal O(k^{\sigma-2})$ for $0<\sigma<1, \sigma=1$, and $1<\sigma<2$, respectively. Under additional summability conditions on the scalarized-gap and inertial contributions, we prove convergence of the full iterate sequence to a weak Pareto solution for $1/2<\sigma<2$.

\item We evaluate the method on nonsmooth multiobjective benchmark problems and image restoration. The experiments examine curvature adaptation at a fixed smoothing level, assess how the smoothing schedule affects termination and computational work, and compare solution quality and computational efficiency.
\end{enumerate}

Section~\ref{sec:prelim} introduces the preliminaries, and Section~\ref{sec:problem} develops the smoothed proximal subproblem and its optimality conditions. Section~\ref{sec:algorithms} presents SAPGM-Backtracking. Section~\ref{sec:sapgm-backtracking-rate} establishes merit-function rates and convergence of the iterates. Section~\ref{sec:experiments} reports numerical experiments, and Section~\ref{sec:conclusion} concludes.

\section{Preliminaries} \label{sec:prelim}

We first recall the vector order, weak Pareto optimality, and the proximal operator used to describe the problem and its updates. We then state the smoothing properties needed to control approximation error and effective curvature. These ingredients are combined in the smoothed proximal subproblem of Section~\ref{sec:problem}.

\subsection{Notation and Basic Definitions}
Let $\mathbb{R}^{n}$ denote the $n$-dimensional Euclidean space. We denote the nonnegative orthant by $\mathbb{R}^{n}_{+}:= \{ v \in \mathbb{R}^{n} \mid v_{i} \geq 0, i=1,\dots,n \}$, and the positive orthant by $\mathbb{R}^{n}_{++}:= \{ v \in \mathbb{R}^{n} \mid v_{i} > 0, i=1,\dots,n \}$. For any vectors $u, v \in \mathbb{R}^{n}$, the partial order induced by $\mathbb{R}^{n}_{+}$ is defined by $u \preceq v$ if and only if $v - u \in \mathbb{R}^{n}_{+}$, and $u \prec v$ if and only if $v-u \in \mathbb{R}^n_{++}$. The standard simplex in $\mathbb{R}^{m}$ is denoted by $\Delta^{m} := \{ \lambda \in \mathbb{R}^{m}_{+} \mid \sum_{i=1}^{m} \lambda_{i} = 1 \}$.
For positive sequences $a_k$ and $b_k$, the notation $a_k=\Theta(b_k)$ means that $c b_k\le a_k\le C b_k$ for some constants $0<c\le C<\infty$ and all sufficiently large $k$.

\begin{definition}[Weak Pareto optimality; see, e.g., Miettinen \cite{miettinen1999nonlinear}]
	An element $x^{*} \in \mathbb{R}^{n}$ is called weakly Pareto optimal if there does not exist $x \in \mathbb{R}^{n}$ such that $F(x) \prec F(x^{*})$, where $F: \mathbb{R}^{n} \to \mathbb{R}^{m}$ is the vector-valued objective function. The set consisting of all weak Pareto solutions $x^{*}$ is called the weak Pareto set and denoted by $X^{*}$.
\end{definition}

For a proper, closed, and convex function $g: \mathbb{R}^n \to \mathbb{R} \cup \{ \infty \}$, the proximal operator \cite{rockafellar1970convex,parikh2014proximal} is defined by
\begin{equation} \label{eq:prox_def}
	\operatorname{prox}_{\lambda g}(x) := \arg\min_{y \in \mathbb{R}^n} \left\{ g(y) + \frac{1}{2\lambda} \| x - y \|^2 \right\},
\end{equation}
where $\lambda > 0$ is a penalty parameter.

\begin{lemma}[Descent Lemma; see, e.g., Nesterov \cite{nesterov2005smooth}] \label{lem:descent}
	Suppose $f: \mathbb{R}^n \to \mathbb{R}$ is continuously differentiable with an $L$-Lipschitz continuous gradient. Then, for any $x, y \in \mathbb{R}^n$, we have
	\begin{equation}
		f(x) \leq f(y) + \langle \nabla f(y), x - y \rangle + \frac{L}{2} \| x - y \|^2.
	\end{equation}
\end{lemma}

The descent lemma applies to differentiable forward terms. To use it for the nonsmooth functions in~\eqref{eq:mop}, we replace them by convex smooth approximations whose error and gradient variation can be controlled through $\mu$.

\begingroup
\begin{definition}[Convex smoothing; cf. \cite{chen2012smoothing,wu2023smoothing}]
\label{def:smoothing}
Let $f:\mathbb R^n\to\mathbb R$ be convex. We use a smoothing function
$\widetilde f:\mathbb R^n\times(0,\infty)\to\mathbb R$ with the following properties:
\begin{enumerate}
\item[(i)] $\widetilde f(\cdot,\mu)$ is convex and continuously differentiable for every $\mu>0$.
\item[(ii)] $\lim_{z\to x,\,\mu\downarrow0}\widetilde f(z,\mu)=f(x)$ for every $x\in\mathbb R^n$.
\item[(iii)] There exists $\kappa>0$ such that, for all $x$ and $\mu_1,\mu_2>0$,
\[
|\widetilde f(x,\mu_2)-\widetilde f(x,\mu_1)|\le\kappa|\mu_2-\mu_1|.
\]
Together with (ii), this gives $|\widetilde f(x,\mu)-f(x)|\le\kappa\mu$.
\item[(iv)] There exists $L^{\rm sm}>0$ such that $\nabla_x\widetilde f(\cdot,\mu)$ is Lipschitz continuous with constant $L^{\rm sm}/\mu$ for every $\mu>0$.
\end{enumerate}
\end{definition}
\endgroup

\section{Smoothed Proximal Subproblem and Optimality Conditions} \label{sec:problem}

This section connects the optimality measure with the step used by the algorithm. We first introduce the merit function and its smoothing-based surrogate, then formulate the regularized min--max proximal subproblem and its optimality conditions. The two estimates established at the end of the section provide the basis for the energy argument in Section~\ref{sec:sapgm-backtracking-rate}.

Throughout the remainder of the paper, we consider the NCMOP introduced in \eqref{eq:mop}, with $F_i=f_i+g_i$, and denote the corresponding smoothed objective by $\tilde{F}_i(x,\mu):=\tilde{f}_i(x,\mu)+g_i(x)$.
The convex analysis uses the following assumptions:
\begin{assumption} \label{assump:main}
	\begin{enumerate}
		\item Each $f_i:\mathbb R^n\to\mathbb R$ is convex and admits a smoothing function satisfying Definition~\ref{def:smoothing}. We use a common error constant $\kappa=\max_i\kappa_i$.
		\item For each $i = 1, \dots, m$, $g_i: \mathbb{R}^n \to \mathbb{R} \cup \{\infty\}$ is a proper, closed, and convex function, which is possibly nonsmooth.
		\item The weakly Pareto solution set $X^\ast$ is nonempty.
	\end{enumerate}
\end{assumption}

To quantify the optimality gap during the algorithmic iterations, following the merit-function framework of Tanabe, Fukuda, and Yamashita \cite{tanabe2024merit}, we use the merit function $u_0: \mathbb{R}^n \to \mathbb{R} \cup \{\infty\}$
\begin{equation}\label{eq:merit}
	u_0(x) := \sup_{z \in \mathbb{R}^{n}} \min_{i =1,\dots,m} [F_{i}(x) - F_{i}(z)].
\end{equation}
The inner minimum measures the improvement shared by all objectives relative to a comparison point $z$, and the supremum selects the largest such common improvement.
By \cite[Theorem 3.1]{tanabe2024merit}, $u_0(x)\ge0$ for $x\in\operatorname{dom}F$, and $u_0(x)=0$ if and only if $x$ is weakly Pareto optimal. For the analysis, we compare the smoothed objectives with the original values $F_i(z)$.
Define the smoothing-based merit surrogate by
\begin{equation}\label{eq:merit-smoothed}
	\tilde{u}_0(x,\mu) := \sup_{z \in \mathbb{R}^{n}} \left\{ \min_{i =1,\dots,m} [\tilde{F}_i(x,\mu) - F_i(z)] + \kappa \mu \right\}.
\end{equation}
For $k \geq 0$ and for all $z \in \mathbb{R}^n$, let $\tilde{\sigma}_k : \mathbb{R}^{n} \to \mathbb{R} \cup \{ - \infty\}$ be defined by
\begin{equation}\label{eq:tilde-sigma}
	\begin{aligned}
		\tilde{\sigma}_k(z) &:= \min_{i = 1,\dotsb,m}[\tilde{F}_{i}(x_{k},\mu_{k}) - F_{i}(z)]+ \kappa \mu_{k}.
	\end{aligned}
\end{equation}
Thus $\widetilde u_0(x_k,\mu_k)=\sup_z\widetilde\sigma_k(z)$. These quantities are used in the analysis. The correction $\kappa\mu$ compensates for the pointwise smoothing error, giving $u_0(x)\le\widetilde u_0(x,\mu)\le u_0(x)+2\kappa\mu$ whenever $u_0(x)$ is finite.

The iterates are generated by a regularized min--max subproblem. Its quadratic coefficient is the effective curvature $\ell_k=L_k/\mu_k$, with $L_k$ denoting the base smoothness estimate.

Specifically, for prescribed values of $x \in \text{dom}(F)$, $y \in \mathbb{R}^{n}$, and an algorithmic parameter $\ell > 0$, we define the subproblem mapping
\begin{equation}\label{4}
	\min_{z \in \mathbb{R}^{n}} \varphi_{\ell}(z;x,y,\mu),
\end{equation}
where the local majorization function is given by
\begin{equation}\label{5}
	\begin{aligned}
			\varphi_\ell(z;x,y,\mu) &:= \max_{i=1,\dotsb,m} \left[\left\langle \nabla \tilde{f}_i (y,\mu),z-y\right\rangle +g_i(z)+\tilde{f}_i(y,\mu)-\tilde{F}_i(x,\mu) \right]  
			+ \frac{\ell }{2} \left\|z-y\right\|^2.
	\end{aligned}
\end{equation}

For $x\in\operatorname{dom}F$, the maximum in~\eqref{5} is proper, lower semicontinuous, and convex. Its positive quadratic regularization is coercive and strongly convex, so the subproblem~\eqref{4} has a unique minimizer; the existence argument is detailed in Theorem~\ref{thm:finite-backtracking}. We denote this minimizer by $p_{\ell}(x,y,\mu)$ and its objective value by $\theta_{\ell}(x,y,\mu)$
\begin{equation}\label{6}
	p_{\ell}(x,y,\mu) := \arg \min_{z\in \mathbb{R}^{n}} \varphi_{\ell}(z;x,y,\mu) \quad \text{and} \quad \theta_{\ell}(x,y,\mu) := \min_{z\in \mathbb{R}^{n}} \varphi_{\ell}(z;x,y,\mu).
\end{equation}

For the subsequent analysis, the representation above also gives the optimality conditions of the smoothed subproblem. Specifically, the first-order optimality conditions associated with \eqref{4} imply that, for all $x \in \text{dom}(F)$ and $y \in \mathbb{R}^{n}$, there exist multipliers $\lambda(x,y,\mu)\in\Delta^m$ and componentwise subgradients $\eta_i(x,y,\mu)\in\partial g_i(p_\ell(x,y,\mu))$, $i=1,\ldots,m$, such that
\begin{equation}\label{7}
	\sum_{i=1}^{m} \lambda_{i}(x,y,\mu) [ \nabla \tilde{f}_{i}(y,\mu) + \eta_{i}(x,y,\mu)] = -\ell [p_{\ell}(x,y,\mu) - y],
\end{equation}
\begin{equation}\label{8}
	\lambda(x,y,\mu) \in \Delta^{m}, \quad \lambda_{j}(x,y,\mu) = 0 \quad \forall j \notin \mathcal{I}(x,y),
\end{equation}
where $\Delta^{m}$ denotes the standard unit simplex, and the active index set is defined as
\begin{equation}
	\mathcal{I}(x,y)
	:=
	\arg\max_{i = 1,\dotsc,m}
	\left[
	\begin{aligned}
		&\left\langle \nabla \tilde{f}_{i}(y,\mu),
		p_{\ell}(x,y,\mu) - y \right\rangle
		+ g_{i}(p_{\ell}(x,y,\mu)) \\
		&\quad
		+ \tilde{f}_{i}(y,\mu)
		- \tilde{F}_{i}(x,\mu)
	\end{aligned}
	\right].
\end{equation}

The subproblem optimality conditions now allow us to compare the accepted point both with the preceding iterate and with a fixed reference point. We keep these two estimates separate because they receive different weights in the accelerated energy argument.

\begin{proposition}\label{prop5.2}
	Let $\tilde{\sigma}_k(z)$ be defined as in \eqref{eq:tilde-sigma}.
	For an accepted step $x_{k+1}=p_{\ell_{k+1}}(x_k,y_k,\mu_{k+1})$ satisfying the componentwise majorization test in Algorithm~\ref{alg:sapgm_back}, the following inequalities hold:
	\begin{equation}\label{15}
		\tilde{\sigma}_{k+1}(z)\leq \tilde{\sigma}_k(z) -\frac{\ell_{k+1}}{2} [2\langle x_{k+1} - y_k, y_k -x_k\rangle + \|x_{k+1} - y_k\|^2]+2\kappa|\mu_k-\mu_{k+1}|,
	\end{equation}
	and
	\begin{equation}\label{16}
		\begin{aligned}
			\tilde{\sigma}_{k+1}(z) \leq& \frac{\ell_{k+1}}{2}[2\langle y_k-x_{k+1},y_k-z\rangle -\|x_{k+1}-y_k\|^2]+2\kappa\mu_{k+1},
		\end{aligned}
	\end{equation}
	where $\ell_{k+1}=L_{k+1}\mu_{k+1}^{-1}$.
\end{proposition}

\begingroup
{\it Proof.}
	\quad We first isolate the estimate obtained with a fixed smoothing parameter.
	Let
	$x^+=p_{\ell_{k+1}}(x_k,y_k,\mu_{k+1})$. Let $\lambda_k\in\Delta^m$ be a
	KKT multiplier of the subproblem at $x^+$, supported on the
	active indices. Since $\min_i a_i\le \sum_i\lambda_{k,i}a_i$,
	define, for a fixed smoothing level $\mu$,
	\[
	\widetilde\sigma_k^\mu(z)
	:=
	\min_{1\le i\le m}
	[\widetilde F_i(x_k,\mu)-F_i(z)]+\kappa\mu .
	\]
	$$
	\tilde\sigma_{k+1}^{\mu_{k+1}}(z)
	\le
	\sum_{i=1}^m\lambda_{k,i}
	[
	\tilde F_i(x^+,\mu_{k+1})-F_i(z)+\kappa\mu_{k+1}
	].
	$$
	Using the smoothing error
	$|\tilde F_i(z,\mu_{k+1})-F_i(z)|\le \kappa\mu_{k+1}$, the descent lemma for
	$\tilde f_i(\cdot,\mu_{k+1})$, and the convexity of $g_i$, we obtain
	$$
\tilde\sigma_{k+1}^{\mu_{k+1}}(z)\le\sum_{i=1}^m\lambda_{k,i}\langle\nabla\tilde f_i(y_k,\mu_{k+1})+\eta_{k,i},x^+-z\rangle+\frac{\ell_{k+1}}{2}\|x^+-y_k\|^2+2\kappa\mu_{k+1},
	$$
	where $\eta_{k,i}\in\partial g_i(x^+)$. The KKT condition gives
	$$
	\sum_{i=1}^m\lambda_{k,i}
	[
	\nabla\tilde f_i(y_k,\mu_{k+1})+\eta_{k,i}
	]
	=
	-\ell_{k+1}(x^+-y_k).
	$$
	Substitution and the identity
	$$
	-2\langle x^+-y_k,x^+-z\rangle+\|x^+-y_k\|^2
	=
	2\langle y_k-x^+,y_k-z\rangle-\|x^+-y_k\|^2
	$$
	yield the inequality
	$$
	\tilde\sigma_{k+1}^{\mu_{k+1}}(z)
	\le
	\frac{\ell_{k+1}}{2}
	[
	2\langle y_k-x^+,y_k-z\rangle-\|x^+-y_k\|^2
	]
	+2\kappa\mu_{k+1}.
	$$
	Taking $x^+=x_{k+1}$, we get \eqref{16}.

	For the recursive estimate, evaluate the same fixed-$\mu_{k+1}$
	subproblem objective at $z=x_k$. Writing $\varphi=\varphi_{\ell_{k+1}}(\cdot;x_k,y_k,\mu_{k+1})$, convexity gives $\varphi(x_k)\le\ell_{k+1}\|x_k-y_k\|^2/2$, while strong convexity gives $\varphi(x^+)\le\varphi(x_k)-\ell_{k+1}\|x^+-x_k\|^2/2$. Accepted componentwise majorization implies $\widetilde F_i(x^+,\mu_{k+1})-\widetilde F_i(x_k,\mu_{k+1})\le\varphi(x^+)$ for every $i$. Taking the minimum relative to the same comparison point and expanding the norm difference gives
	$$
	\tilde\sigma_{k+1}^{\mu_{k+1}}(z)
	\le
	\tilde\sigma_{k}^{\mu_{k+1}}(z)
	-
	\frac{\ell_{k+1}}{2}
	[
	2\langle x_{k+1}-y_k,y_k-x_k\rangle
	+\|x_{k+1}-y_k\|^2
	].
	$$
	Finally, the smoothing-parameter Lipschitz condition implies
	$$
	|\tilde\sigma_k^{\mu_k}(z)-\tilde\sigma_k^{\mu_{k+1}}(z)|
	\le
	2\kappa|\mu_k-\mu_{k+1}|.
	$$
	Thus the dynamic-smoothing version is
	$$
	\tilde{\sigma}_{k+1}(z)\leq
	\tilde{\sigma}_k(z)
	-\frac{\ell_{k+1}}{2}
	[
	2\langle x_{k+1}-y_k,y_k-x_k\rangle
	+\|x_{k+1}-y_k\|^2
	]
	+2\kappa|\mu_k-\mu_{k+1}|.
	$$ \qed
\endgroup

\section{Smoothing Accelerated Proximal Gradient Method with Nonmonotone Backtracking Strategy}\label{sec:algorithms}
\label{subsec:sapgm_back}

We now turn the proximal subproblem of Section~\ref{sec:problem} into an implementable iteration. The construction couples each curvature trial with its own momentum and extrapolation point. After presenting the algorithm, we explain its stopping test and establish finite termination of the inner search together with the identities used in the convergence analysis.

SAPGM-Backtracking adapts the base smoothness estimate $L_k$ to avoid additional conservatism as the smoothing parameter decreases. Each outer iteration starts from $L_{k,0}=L_k/\eta$. The inner search increases the trial base estimate $L_{k,j}$ only when componentwise majorization fails, recomputing momentum, extrapolation, and the proximal candidate for the trial effective curvature $\ell_{k,j}=L_{k,j}/\mu_{k+1}$. Thus nonmonotonicity refers to the accepted sequence $\{L_k\}$.

\begin{algorithm}[!htbp]
\renewcommand{\algorithmicrequire}{\textbf{Input:}}
\renewcommand{\algorithmicensure}{\textbf{Output:}}
\caption{Smoothing Accelerated Proximal Gradient Method with Nonmonotone Backtracking Strategy}
\label{alg:sapgm_back}
\begin{algorithmic}[1]
\REQUIRE $x_0\in\operatorname{dom}F$, tolerances $\varepsilon_R,\varepsilon_\mu>0$,
$\mu_0\in(0,1]$, $L_0>0$, $\eta>1$, and $\sigma\in(0,2]$.
\STATE Set $x_{-1}=x_0$, $\ell_0=L_0/\mu_0$, $t_0=1$, and $k=0$.
\LOOP
\STATE Set $\mu_{k+1}=\mu_0/(k+2)^\sigma$, $L_{k,0}=L_k/\eta$, and $j=0$.
\LOOP
\STATE Set $\ell_{k,j}=L_{k,j}/\mu_{k+1}$ and compute
\[
t_{k+1,j}=\frac{1+\sqrt{1+4(\ell_{k,j}/\ell_k)t_k^2}}{2},\qquad
y_{k,j}=x_k+\frac{t_k-1}{t_{k+1,j}}(x_k-x_{k-1}).
\]
\STATE Compute $\widehat x_{k+1,j}=p_{\ell_{k,j}}(x_k,y_{k,j},\mu_{k+1})$.
\IF{the following holds for every $i=1,\ldots,m$:
\[
\begin{aligned}
\widetilde f_i(\widehat x_{k+1,j},\mu_{k+1})
&\le\widetilde f_i(y_{k,j},\mu_{k+1})
+\langle\nabla\widetilde f_i(y_{k,j},\mu_{k+1}),\widehat x_{k+1,j}-y_{k,j}\rangle\\
&\quad+\frac{\ell_{k,j}}2\|\widehat x_{k+1,j}-y_{k,j}\|^2
\end{aligned}
\]}
\STATE Accept $x_{k+1}=\widehat x_{k+1,j}$, $y_k=y_{k,j}$, $L_{k+1}=L_{k,j}$,
$\ell_{k+1}=\ell_{k,j}$, $t_{k+1}=t_{k+1,j}$, and leave the inner loop.
\ENDIF
\STATE Set $L_{k,j+1}=\eta L_{k,j}$ and $j\leftarrow j+1$.
\ENDLOOP
\IF{$\ell_{k+1}\|x_{k+1}-y_k\|<\varepsilon_R$ and $\mu_{k+1}<\varepsilon_\mu$}
\RETURN $x_{k+1}$.
\ENDIF
\STATE Set $\theta_{k+1}=(t_k-1)/t_{k+1}$ and $k\leftarrow k+1$.
\ENDLOOP
\ENSURE The accepted iterate $x_{k+1}$.
\end{algorithmic}
\end{algorithm}
The stopping test controls the scaled proximal displacement and the smoothing parameter. The latter implies a componentwise approximation error below $\kappa\varepsilon_\mu$; a prescribed error threshold $\varepsilon_{\rm sm}$ is obtained by taking $\varepsilon_\mu=\varepsilon_{\rm sm}/\kappa$. The algorithm allows $\sigma=2$; the merit-rate theorem below concerns $0<\sigma<2$.

\begingroup
\begin{theorem}[Well-definedness and finite termination of backtracking]\label{thm:finite-backtracking}
Under Assumption~\ref{assump:main}, Algorithm~\ref{alg:sapgm_back} with exact subproblem solves is well defined, and every trial subproblem has a unique minimizer in $\operatorname{dom}F$. Let $\overline L=\max_iL_i^{\rm sm}$. The first accepted trial index $J_k\ge0$ and base estimates satisfy
\[
J_k\le\max\!\left\{0,\left\lceil\log_\eta\!\left(\frac{\eta\overline L}{L_k}\right)\right\rceil\right\},\qquad
0<L_{k+1}\le\max\{L_k/\eta,\eta\overline L\},\qquad
L_k\le\max\{L_0,\eta\overline L\}.
\]
Thus each inner search terminates in at most $J_k+1$ trials.
\end{theorem}

{\it Proof.}
Fix $x\in\operatorname{dom}F$, $y\in\mathbb R^n$, and $\mu,\ell>0$. Denote the maximum term in~\eqref{5} by $h(z)$. Each component is a proper lower semicontinuous convex function, and all components are finite at $x$, so $h$ is proper, lower semicontinuous, and convex. A proper closed convex function on a finite-dimensional space admits an affine minorant \cite{rockafellar1970convex}. Applying this to any one component gives $h(z)\ge a^Tz+b$ for some $a,b$. Hence
\[
\varphi_\ell(z;x,y,\mu)\ge a^Tz+b+\frac\ell2\|z-y\|^2\longrightarrow+\infty
\quad\text{as }\|z\|\longrightarrow\infty.
\]
Lower semicontinuity and coercivity ensure attainment, and strong convexity gives uniqueness. Since $\varphi_\ell(x;x,y,\mu)$ is finite, its minimizer has finite values for every $g_i$ and therefore belongs to $\operatorname{dom}F$.

At outer iteration $k$, $\mu_{k+1}>0$ is fixed and $L_{k,j}=\eta^{j-1}L_k$. Definition~\ref{def:smoothing}(iv) and Lemma~\ref{lem:descent} give, for every component and every pair of points,
\[
\widetilde f_i(z,\mu_{k+1})\le\widetilde f_i(y,\mu_{k+1})
+\langle\nabla\widetilde f_i(y,\mu_{k+1}),z-y\rangle
+\frac{\overline L}{2\mu_{k+1}}\|z-y\|^2.
\]
Thus every trial with $L_{k,j}\ge\overline L$ passes the test, even though its extrapolation point changes with $j$. Geometric growth of $L_{k,j}$ yields the stated bound on $J_k$, so there are at most $J_k+1$ subproblem trials. If the first trial is accepted, $L_{k+1}=L_k/\eta$. Otherwise, the preceding rejected trial must have a base estimate below $\overline L$, and hence $L_{k+1}<\eta\overline L$. Induction gives the uniform upper bound. Finally, positive finite curvature values make the momentum recursion finite and real, and each accepted point stays in $\operatorname{dom}F$. Starting from $x_{-1}=x_0\in\operatorname{dom}F$ proves well-definedness by induction.\qed
\endgroup

\begingroup
The theorem concerns the inner search; a uniform positive lower bound on $L_k$ is a separate condition in the rate analysis.
For every accepted step, the momentum and extrapolation satisfy
\[
\begin{aligned}
\ell_{k+1}^{-1}t_{k+1}(t_{k+1}-1)=\ell_k^{-1}t_k^2,\quad
t_{k+1}y_k-(t_{k+1}-1)x_k=t_kx_k-(t_k-1)x_{k-1}.
\end{aligned}
\]
\endgroup

\section{Convergence Analysis of SAPGM-Backtracking}\label{sec:sapgm-backtracking-rate}

We first establish merit-function rates from the accepted-step estimates, then prove convergence of the iterates by quasi-Fej\'er monotonicity.

\subsection{Convergence rates via a merit function}\label{subsec:merit-rates}

Define the level set
\[
\mathcal L_F(c):=\{x\in\mathbb R^n:\ F(x)\preceq c\}.
\]

\begin{assumption}\label{a1}
\leavevmode
\begin{itemize}
\item[(a)] There is $c\in\mathbb R^m$ such that $K:=\mathcal L_F(c)$ is nonempty and compact, and $x_k\in K$ for all $k$. Set $S:=X^\ast\cap K$ and $D_S:=\sup_{z\in S}\|x_0-z\|^2$.
\item[(b)] The base smoothness estimates satisfy $0<L_{\min}\le L_k\le L_{\max}<\infty$.
\end{itemize}

\end{assumption}

Under part (a), $S\ne\varnothing$, $D_S<\infty$, and, for $x\in K$,
\[
u_0(x)=\max_{z\in S}\min_{1\le i\le m}[F_i(x)-F_i(z)].
\]
Indeed, for fixed $x\in K$, the function $h_x(z):=\min_i[F_i(x)-F_i(z)]$ is upper semicontinuous on the nonempty compact set $\mathcal L_F(F(x))\subset K$, because each $f_i$ is continuous and each $g_i$ is lower semicontinuous. All components of $F$ are finite on $K$. Since $h_x(x)=0$ and $h_x(z)<0$ outside $\mathcal L_F(F(x))$, the global supremum is attained at some $z_x$ in that level set. If $z_x$ were not weakly Pareto optimal, a point $y$ with $F(y)\prec F(z_x)\preceq F(x)$ would belong to $K$ and satisfy $h_x(y)>h_x(z_x)$, a contradiction. Hence $z_x\in S$. Boundedness of $K$ also gives $D_S<\infty$.

The level-set representation fixes the comparison set for the merit estimate. We next study the momentum growth while allowing the accepted effective curvature to vary between iterations.
Recall $\widetilde\sigma_k$ from \eqref{eq:tilde-sigma}. Throughout this section,
write
\[
\ell_k=\frac{L_k}{\mu_k},\qquad
\gamma_k=\ell_k^{-1}=\frac{\mu_k}{L_k},\qquad
s_k=t_k\sqrt{\gamma_k}.
\]
The smoothing sequence is non-increasing. If $\ell_{k+1}\ge\ell_k$, then
$\gamma_{k+1}\le\gamma_k$; under \Cref{a1},
$\mu_k/L_{\max}\le\gamma_k\le\mu_k/L_{\min}$.

\Needspace{9\baselineskip}
\begin{lemma}\label{tk}\label{lem:sigma_2_bounds}
Let $\mu_k=\mu_0/(k+1)^\sigma$, with $0<\sigma\le2$, and suppose
\Cref{a1} holds. For the momentum sequence $\{t_k\}_{k\ge0}$ generated by Algorithm~\ref{alg:sapgm_back}:
\begin{itemize}
\item[(i)]
\[
s_k+\frac12\sqrt{\gamma_{k+1}}
\le s_{k+1}
\le s_k+\sqrt{\gamma_{k+1}}.
\]
\item[(ii)]
\[
\gamma_{k+1}t_{k+1}(t_{k+1}-1)=\gamma_kt_k^2.
\]
\item[(iii)] If $0<\sigma<2$, then there are constants
$0<c_\sigma\le C_\sigma<\infty$ such that
\[
c_\sigma(k+1)^{1-\sigma/2}\le s_k
\le C_\sigma(k+1)^{1-\sigma/2}
\]
for all sufficiently large $k$. If $\sigma=2$, then
$s_k=\Theta(\log(k+1))$.
\item[(iv)]
\[
\sum_{j=1}^{k}L_j^{-1}t_j\mu_j^2
=
\begin{cases}
O(k^{2-2\sigma}),&0<\sigma<1,\\
O(\log k),&\sigma=1,\\
O(1),&1<\sigma\le2.
\end{cases}
\]
\item[(v)] If $\ell_{k+1}\ge\ell_k$, then
$\theta_{k+1}^2<\ell_k/\ell_{k+1}\le1$.
\end{itemize}
\end{lemma}

\begingroup
{\it Proof.}
	\quad Put $r_k=\ell_{k+1}/\ell_k$. The momentum recursion and
	$\sqrt{1+4a}\in[2\sqrt a,1+2\sqrt a]$ give
	\[
	\frac12+\sqrt{r_k}\,t_k
	\le t_{k+1}
	\le 1+\sqrt{r_k}\,t_k.
	\]
	Multiplying by $\sqrt{\gamma_{k+1}}$ and using
	$\gamma_{k+1}r_k=\gamma_k$ proves (i). Squaring the defining recursion
	gives (ii).

	Iterating (i) yields
	\[
	s_0+\frac12\sum_{j=1}^{k}\sqrt{\gamma_j}
	\le s_k
	\le s_0+\sum_{j=1}^{k}\sqrt{\gamma_j}.
	\]
	By \Cref{a1},
	\[
	\sqrt{\frac{\mu_j}{L_{\max}}}
	\le\sqrt{\gamma_j}
	\le\sqrt{\frac{\mu_j}{L_{\min}}}.
	\]
	Integral comparison for
	$\sum_{j=1}^{k}(j+1)^{-\sigma/2}$ proves (iii), including the
	logarithmic case $\sigma=2$.
	
	Next,
	\[
	L_j^{-1}t_j\mu_j^2
	=
	s_j\frac{\mu_j^{3/2}}{\sqrt{L_j}}.
	\]
	For $0<\sigma<2$, part (iii) bounds this term by
	$C(j+1)^{1-2\sigma}$; for $\sigma=2$, it is bounded by
	$C\log(j+1)/(j+1)^3$. Summing these estimates proves (iv).

\begin{samepage}
	Finally, the lower bound for $t_{k+1}$ implies
	$t_{k+1}^2\ge(\ell_{k+1}/\ell_k)t_k^2$. Using
	$\theta_{k+1} = (t_k-1)/t_{k+1} < t_k/t_{k+1}$, we obtain:
	$$
	\theta_{k+1}^2 < \frac{t_k^2}{t_{k+1}^2} \leq \frac{\ell_k}{\ell_{k+1}} \le 1,
	$$
	which proves (v).
\qed
\end{samepage}
\endgroup

\begingroup
Lemma~\ref{tk} gives the growth of the accelerated scaling. We now combine it with the accepted-step estimates to track the merit quantity and the accumulated smoothing error.
The recursive inequality~\eqref{15} links consecutive smoothed merit quantities,
whereas~\eqref{16} compares the new iterate with a reference point. Their weighted
combination gives the following energy estimate.

\Needspace{5\baselineskip}
\begin{theorem}\label{thm:energy-merit-template}
Suppose that Assumptions~\ref{assump:main} and~\ref{a1} hold and that
$\{x_k\}$ is generated by Algorithm~\ref{alg:sapgm_back}. For $z\in S$, set
\[
w_k(z)=t_kx_k-(t_k-1)x_{k-1}-z,\qquad
\mathcal E_k(z)=2\ell_k^{-1}t_k^2\widetilde\sigma_k(z)+\|w_k(z)\|^2.
\]
Then
\begin{equation}\label{eq:back-energy}
\mathcal E_{k+1}(z)\le\mathcal E_k(z)+E_{k+1},
\end{equation}
where
\[
E_{k+1}=4\kappa\ell_{k+1}^{-1}t_{k+1}(t_{k+1}-1)(\mu_k-\mu_{k+1})+4\kappa\ell_{k+1}^{-1}t_{k+1}\mu_{k+1}.
\]
Consequently, there is a finite constant $C_E$, independent of $k$, such that
\begin{equation}\label{eq:back-merit-energy}
u_0(x_k)\le\frac{C_E+\sum_{j=1}^k E_j}{2\ell_k^{-1}t_k^2}.
\end{equation}
\end{theorem}

\noindent{\it Proof.}
Multiply~\eqref{15} by $2\ell_{k+1}^{-1}t_{k+1}(t_{k+1}-1)$ and
\eqref{16} by $2\ell_{k+1}^{-1}t_{k+1}$. Adding gives
\begin{equation}\label{reswkwk1}
\begin{aligned}
&2\ell_{k+1}^{-1}t_{k+1}(t_{k+1}-1)\widetilde\sigma_k(z)
-2\ell_{k+1}^{-1}t_{k+1}^2\widetilde\sigma_{k+1}(z)\\
&\quad\ge 2t_{k+1}\langle x_{k+1}-y_k,
t_{k+1}y_k-(t_{k+1}-1)x_k-z\rangle+t_{k+1}^2\|x_{k+1}-y_k\|^2-E_{k+1}.
\end{aligned}
\end{equation}
The accepted-step identities and
$w_{k+1}(z)=w_k(z)+t_{k+1}(x_{k+1}-y_k)$ turn the right-hand side into
$\|w_{k+1}(z)\|^2-\|w_k(z)\|^2-E_{k+1}$, proving~\eqref{eq:back-energy}.
Telescoping and dropping the final squared norm give
\[
2\ell_k^{-1}t_k^2\widetilde\sigma_k(z)
\le\mathcal E_0(z)+\sum_{j=1}^k E_j.
\]
Since $t_0=1$ and $x_{-1}=x_0$, the quantity
$C_E:=\sup_{z\in S}\mathcal E_0(z)$ is finite: $D_S<\infty$ and each
lower-semicontinuous component $F_i$ is bounded below on the compact set $K$.
The pointwise smoothing bound gives
$\min_i[F_i(x_k)-F_i(z)]\le\widetilde\sigma_k(z)$.
Taking the supremum over $S$ and using its representation of $u_0(x_k)$ proves
\eqref{eq:back-merit-energy}. \qed

The energy estimate reduces the rate calculation to comparing the partial sum of $E_j$ with the growth of $\ell_k^{-1}t_k^2$. Lemma~\ref{tk} supplies both orders for each smoothing sequence.

\Needspace{5\baselineskip}
\begin{theorem}\label{thm:sapgm-back-rate}
Suppose that Assumptions~\ref{assump:main} and~\ref{a1} hold,
$\{x_k\}$ is generated by Algorithm~\ref{alg:sapgm_back}, and
$\mu_k=\mu_0/(k+1)^\sigma$ with $0<\sigma<2$. There is a constant $C>0$,
independent of $k$, such that, for all $k\ge2$,
\[
u_0(x_{k+1})\le
\begin{cases}
Ck^{-\sigma},&0<\sigma<1,\\
C(\log k)/k,&\sigma=1,\\
Ck^{\sigma-2},&1<\sigma<2.
\end{cases}
\]
\end{theorem}

\noindent{\it Proof.}
Lemma~\ref{tk}(iii) and Assumption~\ref{a1}(b) give
$t_{k+1}=O(k+1)$ and
$(\mu_k-\mu_{k+1})/\mu_{k+1}=O((k+1)^{-1})$. Thus the perturbation in
Theorem~\ref{thm:energy-merit-template} satisfies
\[
E_{k+1}\le C\ell_{k+1}^{-1}t_{k+1}\mu_{k+1}
=CL_{k+1}^{-1}t_{k+1}\mu_{k+1}^2.
\]
By Lemma~\ref{tk}(iii)--(iv),
\[
\ell_k^{-1}t_k^2=s_k^2=\Theta(k^{2-\sigma}),\qquad
C_E+\sum_{j=1}^kE_j=
\begin{cases}
O(k^{2-2\sigma}),&0<\sigma<1,\\
O(\log k),&\sigma=1,\\
O(1),&1<\sigma<2.
\end{cases}
\]
Substitution into~\eqref{eq:back-merit-energy} yields the stated rates. \qed

The three regimes describe the balance between smoothing accuracy and
accelerated progress. For $0<\sigma<1$, the merit bound has the same order as
the smoothing bias. At $\sigma=1$, the accumulated smoothing perturbation
grows logarithmically, giving the transition rate $O((\log k)/k)$. For
$1<\sigma<2$, that perturbation has a bounded sum, and the energy scaling
$\ell_k^{-1}t_k^2=\Theta(k^{2-\sigma})$ determines the bound. These estimates
describe asymptotic merit decay; finite-budget work also depends on when
the smoothing parameter crosses its stopping threshold. A larger $\sigma$
can therefore permit earlier numerical termination without a sharper
asymptotic merit bound. Subsection~\ref{sec:sequential-convergence} establishes convergence of the iterates.
\endgroup

\begingroup
\subsection{Full-sequence convergence}\label{sec:sequential-convergence}

\begingroup
We now study convergence of the iterates by quasi-Fej\'er monotonicity. The argument bounds increases in the squared distance to each weak Pareto comparison point by a summable sequence.

\Needspace{8\baselineskip}
\begin{definition}[Quasi-Fej\'er monotonicity]\label{def:qf-distance}
Let $C\subset\mathbb R^n$ be nonempty. A sequence $\{x_k\}$ is uniformly quasi-Fej\'er monotone with respect to $C$ if there exists $\varepsilon_k\ge0$, independent of $z\in C$, such that
\begin{equation}\label{eq:qf-distance-definition}
\|x_{k+1}-z\|^2\le\|x_k-z\|^2+\varepsilon_k
\quad(z\in C),\qquad \sum_{k=0}^{\infty}\varepsilon_k<\infty.
\end{equation}
\end{definition}

The summable allowance in~\eqref{eq:qf-distance-definition} accommodates the
inertial displacement and smoothing error. Its role in full-sequence
convergence follows from the elementary distance argument below.

\begin{samepage}
\begin{lemma}\label{lem:qf-distance-convergence}
If $\{x_k\}$ satisfies Definition~\ref{def:qf-distance}, then it is bounded
and $\|x_k-z\|$ converges for every $z\in C$. If the sequence has a cluster
point $\bar x\in C$, then $x_k\to\bar x$.
\end{lemma}

\noindent{\it Proof.}
For a fixed $z\in C$, define
\[
d_k(z)=\|x_k-z\|^2+\sum_{j=k}^{\infty}\varepsilon_j.
\]
Inequality~\eqref{eq:qf-distance-definition} implies
$0\le d_{k+1}(z)\le d_k(z)$, so $d_k(z)$ has a limit.
Since the error tail tends to zero, $\|x_k-z\|^2$ also has a limit.
Choosing any $z\in C$ gives boundedness. If $x_{k_j}\to\bar x\in C$,
the limit of $\|x_k-\bar x\|$ must be zero, proving convergence of the
whole sequence. \qed
\end{samepage}

We apply this principle with $C=S=X^*\cap K$ from Assumption~\ref{a1}.
The proximal estimate compares $x_{k+1}$ with $y_k$; expanding
$y_k=x_k+\theta_{k+1}(x_k-x_{k-1})$ then gives a distance estimate from
$x_k$. The following conditions make the scalarized-gap and inertial
contributions summable. The smoothing contribution is summable for
$\sigma>1/2$.
\endgroup

\begin{theorem}\label{thm:qf-sequential-convergence}\label{ass:qf-back}
Let Assumptions~\ref{assump:main} and~\ref{a1} hold, and let
$\{x_k\}$ be generated by Algorithm~\ref{alg:sapgm_back} with
$\mu_k=\mu_0/(k+1)^\sigma$, where $1/2<\sigma<2$.
Let $\lambda_k\in\Delta^m$ be the accepted subproblem multipliers from~\eqref{7}--\eqref{8}.
Suppose that a nonnegative sequence $\{\rho_k\}$, independent of $z\in S$, satisfies
\begin{equation}\label{eq:scalarized-gap}
\sum_{i=1}^m\lambda_{k,i}[F_i(x_{k+1})-F_i(z)]\ge-\rho_k,
\qquad z\in S,
\end{equation}
and
\begin{equation}\label{eq:qf-summability}
\sum_{k=0}^\infty\ell_{k+1}^{-1}\rho_k<\infty,
\qquad
\sum_{k=0}^\infty\theta_{k+1}\|x_k-x_{k-1}\|<\infty.
\end{equation}
Then $\{x_k\}$ converges to a weak Pareto solution in $S$.
\end{theorem}

\noindent{\it Proof.}
We first verify the distance estimate in Definition~\ref{def:qf-distance}, and then identify a cluster point in $S$ to apply Lemma~\ref{lem:qf-distance-convergence}.
The accepted majorization test, the subproblem optimality condition, and
the smoothing bound give, for every $z\in S$,
\[
\sum_{i=1}^m\lambda_{k,i}[F_i(x_{k+1})-F_i(z)]
\le\frac{\ell_{k+1}}2\bigl(\|y_k-z\|^2-\|x_{k+1}-z\|^2\bigr)
+2\kappa\mu_{k+1}.
\]
Combining this with~\eqref{eq:scalarized-gap} yields
\begin{equation}\label{eq:qf-core}
\|x_{k+1}-z\|^2\le\|y_k-z\|^2
+2\ell_{k+1}^{-1}\rho_k+4\kappa\ell_{k+1}^{-1}\mu_{k+1}.
\end{equation}
Put $a_k=\theta_{k+1}\|x_k-x_{k-1}\|$ and let $D$ be the diameter of $K$.
Since $x_k,z\in K$ and $\|y_k-x_k\|=a_k$,
\[
\|x_{k+1}-z\|^2\le\|x_k-z\|^2+\varepsilon_k,\qquad
\varepsilon_k=2Da_k+a_k^2+2\ell_{k+1}^{-1}\rho_k
+4\kappa\ell_{k+1}^{-1}\mu_{k+1}.
\]
The assumptions give $\sum_k a_k<\infty$, hence $\sum_k a_k^2<\infty$.
Also, $\ell_{k+1}^{-1}\mu_{k+1}=\mu_{k+1}^2/L_{k+1}$ is summable because
$\sigma>1/2$ and $L_{k+1}\ge L_{\min}$. Thus $\sum_k\varepsilon_k<\infty$.
The displayed distance estimate has the form of Definition~\ref{def:qf-distance}, and $\{x_k\}$ is quasi-Fej\'er monotone with respect to $S$.
Compactness of $K$ gives a cluster point $\bar x\in K$.
Theorem~\ref{thm:sapgm-back-rate} gives $u_0(x_k)\to0$; lower semicontinuity
of $u_0$ \cite{tanabe2024merit} then yields $u_0(\bar x)=0$, so $\bar x\in S$.
Finally, Lemma~\ref{lem:qf-distance-convergence} gives convergence of $\|x_k-\bar x\|$. Its convergent subsequence forces that limit to be zero, so the entire sequence converges to $\bar x$. \qed
\endgroup

\section{Numerical Experiments}\label{sec:experiments}
\begingroup
\begingroup
We compare Pareto-front quality and computational work, study curvature adaptation and smoothing parameter choices, and test the method on image restoration. All results use Python on an Intel Core i7-11700 workstation, with four workers and one numerical-library thread per worker. We test 50\% and 20\% curvature reductions: the initial trial is $L_{k,0}=\rho L_k$, with $\rho=0.5$ or $0.8$, while rejected trials are enlarged by $\eta=2$. The 50\% strategy agrees with Algorithm~\ref{alg:sapgm_back}; the 20\% strategy changes only the initial reduction.

The comparison methods are ESGDM \cite{gebken2021efficient}, MPB \cite{makela2014mpb}, and MSGDB \cite{montonen2017msgdb}. Our MSGDB implementation retains center cuts and checks the joint direction and model error. Image experiments include scalar FISTA for the first objective. Implementation details and all solver settings accompany the data; algorithm code is available at \url{https://github.com/Huanghcz/SAPGM-Backtracking-Code}.
\endgroup

\subsection{Benchmark quality and computational work}
The suite contains 18 problems from M\"{a}kel\"{a} et al.\ \cite{makela2014mpb}, eight synthetic learning-loss problems, five sparse problems, and three composite three-objective problems, with 200 common starts for each problem. Table~\ref{tab:test_problems} lists the families. Nonconvex classical and activation-loss cases are included as numerical stress tests, outside the convex convergence statements. The aggregate indicators follow the \texttt{zfista} implementation associated with Tanabe et al.\ \cite{tanabe2023accelerated}, including hypervolume, purity, and spread criteria \cite{custodio2011direct,lapucci2023memetic}.
\endgroup
\begin{table}[!htbp]
  \centering
  \caption{The 34 benchmark problems used in the aggregate comparison.}
  \label{tab:test_problems}
  \scriptsize
  \setlength{\tabcolsep}{3.5pt}
  \begin{tabular}{clp{0.63\textwidth}}
    \toprule
    Index & Family & Problem names \\
    \midrule
    1--18 & M\"{a}kel\"{a}~\cite{makela2014mpb} &
    \texttt{CB3\_DEM}, \texttt{CB3\_QL}, \texttt{CB3\_LQ},
    \texttt{CB3\_Mif1}, \texttt{CB3\_Wolfe}, \texttt{DEM\_QL},
    \texttt{DEM\_LQ}, \texttt{DEM\_Mif1}, \texttt{DEM\_Wolfe},
    \texttt{QL\_LQ}, \texttt{QL\_Mif1}, \texttt{QL\_Wolfe},
    \texttt{LQ\_Mif1}, \texttt{LQ\_Wolfe}, \texttt{Mif1\_Wolfe},
    \texttt{Cres\_Mif2}, \texttt{Mif2\_WF}, and \texttt{Mif2\_Spi}. \\
    \addlinespace
    19--26 & Synthetic learning & \texttt{L1\_vs\_Hinge}, \texttt{Pinball\_vs\_ReLU},
    \texttt{Eps\_vs\_Perceptron}, \texttt{PosErr\_vs\_NegErr},
    \texttt{LeakyReLU\_vs\_LargeHinge}, \texttt{LAD}, \texttt{Pinball},
    and \texttt{Hinge}. \\
    \addlinespace
    27--31 & Synthetic sparse & \texttt{500x100\_Sp0.8},
    \texttt{500x100\_Sp0.6}, \texttt{500x100\_Sp0.4},
    \texttt{500x100\_Sp0.2}, and \texttt{500x100\_Sp0.1}. \\
    \addlinespace
    32--34 & Triobjective composite & \texttt{CE\_MDMP\_3},
    \texttt{CE\_RHMDMP\_3}, and \texttt{CE\_AMDMP\_3}. \\
    \bottomrule
  \end{tabular}
\end{table}
\begingroup

\begingroup
\begingroup
\begingroup
Figure~\ref{fig:benchmark-six-profiles} compares nine configurations: both reductions with $\sigma\in\{1,1.5,2\}$, ESGDM \cite{gebken2021efficient}, MPB \cite{makela2014mpb}, and MSGDB \cite{montonen2017msgdb}. The common domain excludes problems 16--18, where MSGDB is inapplicable. Each problem has 200 common starts, a 10,000-total-step budget, and a 120-second limit, with $L_0=0.01$, $\eta=2$, and SAPGM residual/smoothing tolerances $10^{-3}$. Following \texttt{zfista} \cite{tanabe2023accelerated}, iteration and time indicators average native successful runs; these subsets may differ. Front metrics use all finite terminal vectors and the pooled nondominated array from all nine configurations. HV is unnormalized, with the pooled coordinatewise maximum as reference; purity uses exact intersections and retains pooled duplicates. Spreads follow the same reference implementation. Ratios compare costs and spreads with their minima, and HV and purity with their maxima. All profiles retain 31 problems in the denominator, including undefined ratios. Six singleton pooled fronts give a $25/31$ spread plateau and zero-HV ties. Full definitions and all 61,200 outcomes accompany the data.
\endgroup

\begingroup
At $\sigma=1.5$, the 50\% and 20\% strategies attain highest purity, including ties, on 16 and 18 problems, with median relative HV of 0.9964 and 0.9946 on the 25 nondegenerate problems. Their mean times are below MSGDB on 21 and 22 of 26 comparable problems, although ESGDM is often faster and spread comparisons are mixed. Figure~\ref{fig:pareto-fronts-selected} shows both verified fronts.
\endgroup


Verified success requires native termination and a common subgradient-direction check at $10^{-3}$. Over 6200 starts per configuration, the 50\%/20\% strategies attain success rates of 93.5\%/93.6\%, 94.2\%/96.7\%, and 79.1\%/78.1\% for $\sigma=1$, $1.5$, and $2$, respectively; the rates for ESGDM, MPB, and MSGDB are 86.2\%, 81.3\%, and 74.8\%.

On joint native-success starts, 20\% reduction lowers mean time by 48.2\% and 31.7\% for $\sigma=1$ and $1.5$, but raises it by 37.0\% for $\sigma=2$. Accepted iterations increase by 17.7\% at $\sigma=1.5$.
\endgroup
\endgroup
\endgroup

\begin{figure}[!htbp]\centering
\includegraphics[width=\textwidth,height=0.79\textheight,keepaspectratio]{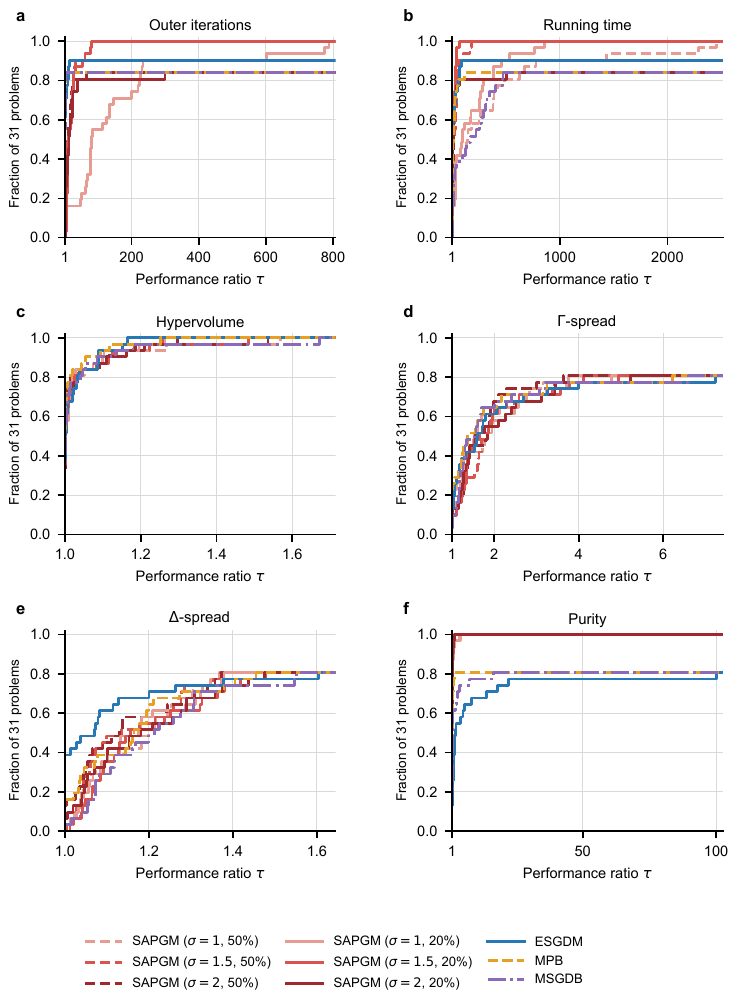}
\caption{Performance profiles of outer iterations, running time, hypervolume, $\Gamma$-spread, $\Delta$-spread, and purity.}
\label{fig:benchmark-six-profiles}\label{fig:section9-main34}
\end{figure}
\begin{figure}[p]\centering
\includegraphics[width=\textwidth,height=0.80\textheight,keepaspectratio]{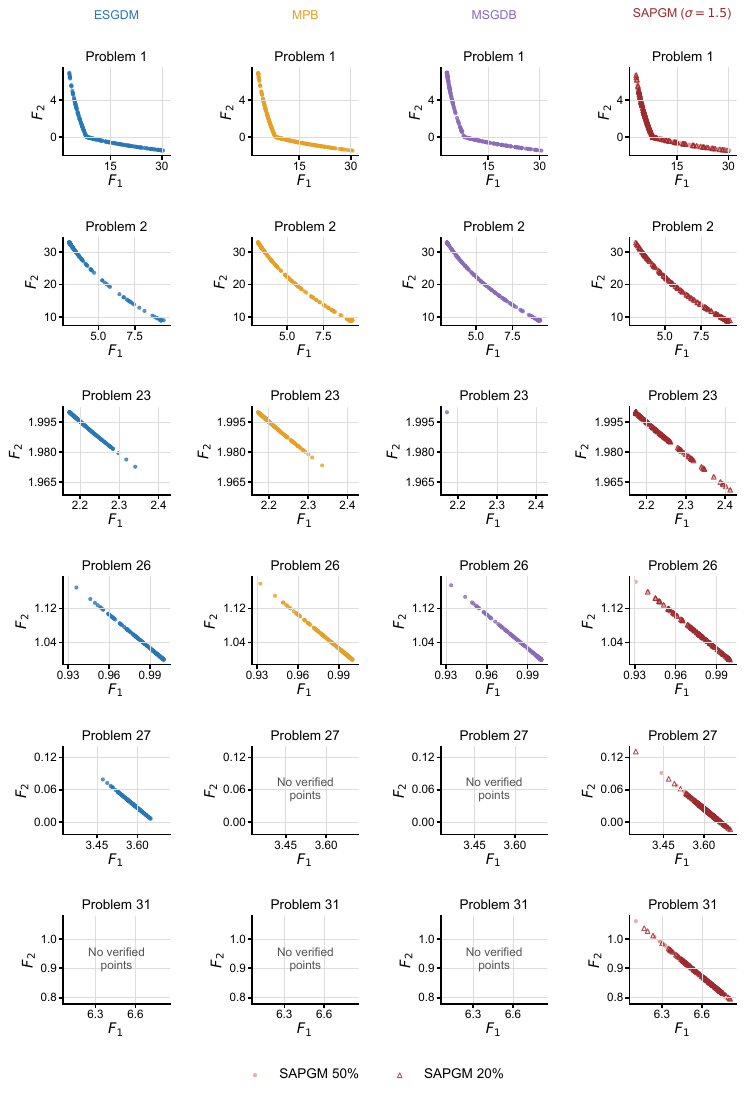}
\caption{Selected two-objective fronts for problems 1, 2, 23, 26, 27, and 31 (rows). SAPGM panels show both reductions at $\sigma=1.5$: filled circles for 50\% and open triangles for 20\%. Points are unique nondominated verified terminal vectors from 200 common starts. Empty panels indicate no verified successes. Coordinates are shared within each row; no points are interpolated.}
\label{fig:pareto-fronts-selected}\label{fig:back-pareto-two}\label{fig:learning-singleton-fronts}
\end{figure}
\begin{figure}[p]\addtocounter{figure}{-1}\centering
\includegraphics[width=\textwidth,height=0.80\textheight,keepaspectratio]{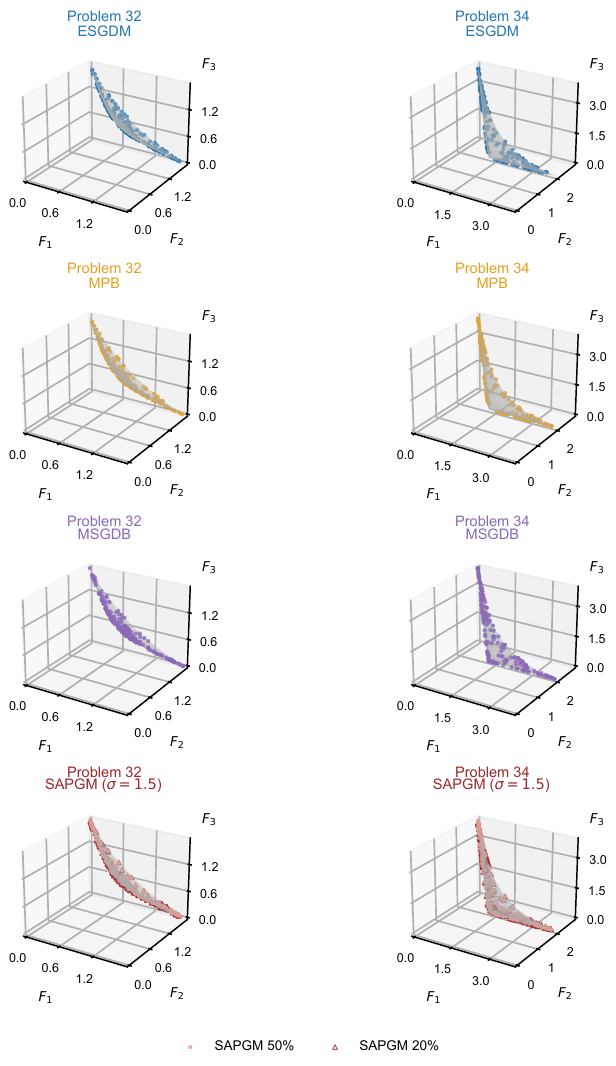}
\caption{Continued: three-objective fronts for problems 32 (left) and 34 (right). Rows show ESGDM, MPB, MSGDB, and both SAPGM reductions at $\sigma=1.5$. Filled circles denote 50\% and open triangles 20\% in the SAPGM row; gray surfaces are analytic reference fronts. Points are verified terminal vectors; coordinate limits and viewing angles are shared within each problem.}
\label{fig:back-pareto-three}\label{fig:pareto-fronts-triobjective}
\end{figure}

\FloatBarrier
\subsection{Curvature adaptation with a fixed smoothing parameter}
We isolate curvature adaptation at $\mu=5\times10^{-4}$ on four convex learning problems (L1-vs-Hinge, LAD, Pinball, and Hinge), using 200 common starts per problem. The two SAPGM reduction strategies, nondecreasing backtracking, and APGMO \cite{tanabe2023accelerated} share the same dual solver, initial base curvature $L_0=0.01$, and residual threshold $10^{-3}$. All 3,200 runs reach the threshold and pass an independent terminal proximal-residual check at the common effective curvature $20$. Paired terminal objective vectors agree at recorded precision.

Figure~\ref{fig:fixedmu-back} shows that both reductions save work relative to nondecreasing backtracking and APGMO. The 50\% and 20\% strategies average 8.95 and 16.86 iterations, with 16.30 and 31.07 primitive proximal calls, respectively; the corresponding counts are 60.22 and 106.52 for nondecreasing backtracking and 60.92 and 108.74 for APGMO. Relative to APGMO, the mean paired running-time ratios are 0.161 and 0.289. The 50\% reduction uses less work, and both strategies reach the same terminal objective vectors. Iteration counts refer to accepted outer steps.

\begin{figure}[!htbp]\centering
\includegraphics[width=\textwidth]{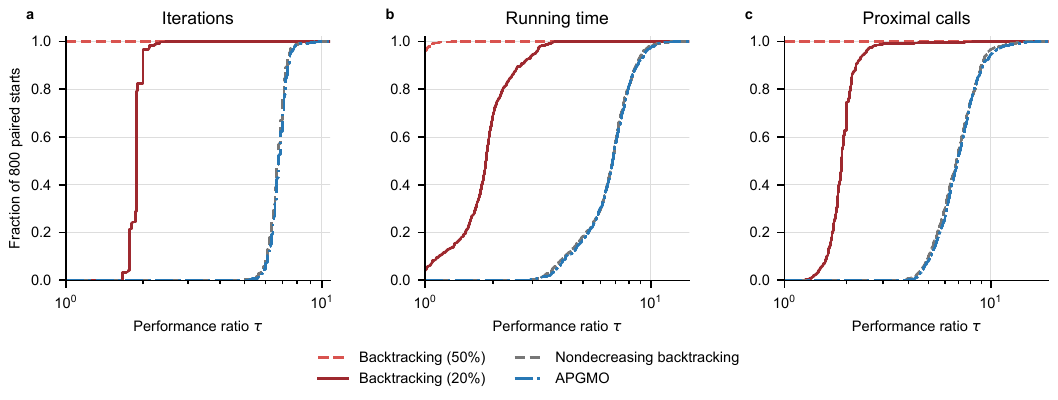}
\caption{Performance profiles with a fixed smoothing parameter over 800 common starts on four learning problems. Each cost is divided by the smallest cost among the two reductions, nondecreasing backtracking, and APGMO. All four configurations satisfy the common residual threshold, with matching terminal objective vectors.}
\label{fig:fixedmu-back}
\end{figure}

\Needspace{6\baselineskip}\subsection{Choice of the smoothing parameter}
\begingroup
The choice of $\sigma$ determines how quickly the smoothing parameter reaches the stopping threshold. The implementation numbers accepted steps by $j=1,2,\ldots$ and uses $\mu^{(j)}=j^{-\sigma}$; the earliest eligible stopping steps are therefore 1001, 101, and 32 for $\sigma=1$, $1.5$, and $2$, before counting rejected curvature trials. Figure~\ref{fig:back-sigma-two} compares verified success and mean total work across the four families. The base settings $L_0=0.01$ and $\eta=2$ are fixed throughout; $\sigma=2$ is a numerical boundary test outside the rate theorem's range $0<\sigma<2$.

All six SAPGM configurations achieve verified success on all 1600 learning starts and all 600 three-objective starts. On the sparse family, changing from 50\% to 20\% reduction raises verified successes from 699 to 716 for $\sigma=1$ and from 674 to 826 for $\sigma=1.5$. Both $\sigma=2$ configurations exhaust the total-step budget on all 1000 starts. Slower smoothing therefore preserves progress when rapid decay makes the effective curvature large. MPB encounters QP failures on this family, and MSGDB has no verified successes.
\endgroup
\begin{figure}[htbp]\centering
\includegraphics[width=\textwidth]{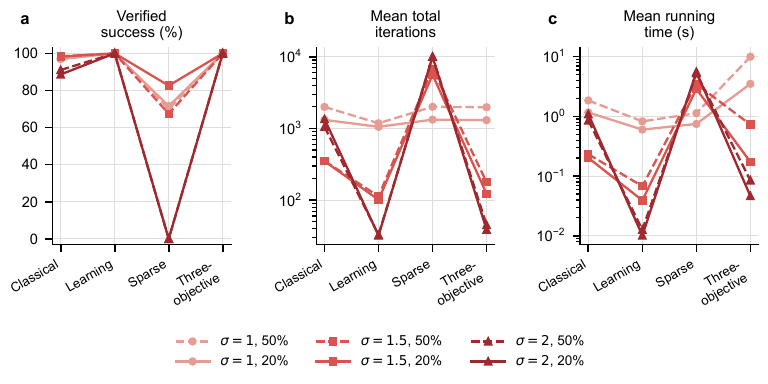}
\caption{Both curvature reductions across the classical, learning, sparse, and three-objective families (3600, 1600, 1000, and 600 starts per configuration). Dashed and solid lines denote 50\% and 20\% reductions. Success requires native termination and the common quality check; work and time are arithmetic means over all attempts, including failures.}\label{fig:back-sigma-two}\label{fig:back-sigma-three}
\end{figure}
\FloatBarrier
\subsection{Image restoration}

\begingroup

For robust observations, replace squared fidelity by an $\ell_1$ loss:
\[
\min_c(F_1(c),F_2(c)),\qquad F_i(c)=\operatorname{mean}|H_iW^{-1}c-b_i|+\lambda_i\|c\|_1.
\]
The observations contain Gaussian noise and sparse outliers, respectively. Selective smoothing applies to nonsmooth fidelity, whereas the coefficient penalty retains soft thresholding. With $\lambda_1=\lambda_2=2\times10^{-5}$ and 1500 iterations, Figure~\ref{fig:robust_l1_restoration_images} displays the highest-PSNR start for each method. Across three starts, the 50\% and 20\% strategies attain PSNR $30.107\pm0.348$ and $30.082\pm0.362$ dB (mean $\pm$ SD), respectively, with mean SSIM of 0.8995 and 0.8996. Their mean running times are 162.28 and 157.12 seconds, compared with 266.98 for MPB and 2981.44 for ESGDM.
\endgroup

\begin{figure}[p]\centering
\includegraphics[width=0.97\textwidth]{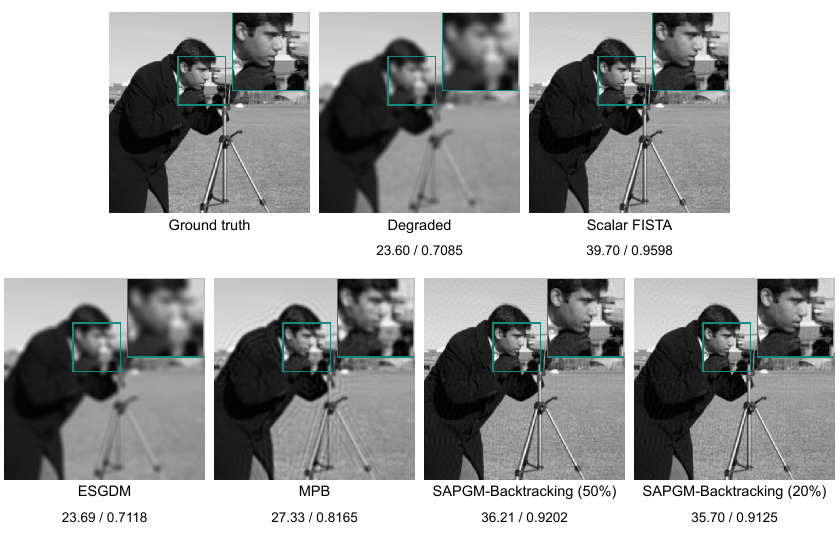}
\caption{Cameraman restoration after 500 iterations with both curvature reductions. Values are PSNR (dB) / SSIM; insets enlarge the same region. Scalar FISTA solves the first objective.}\label{fig:cameraman_restoration}
\vspace{8pt}
\includegraphics[width=0.97\textwidth]{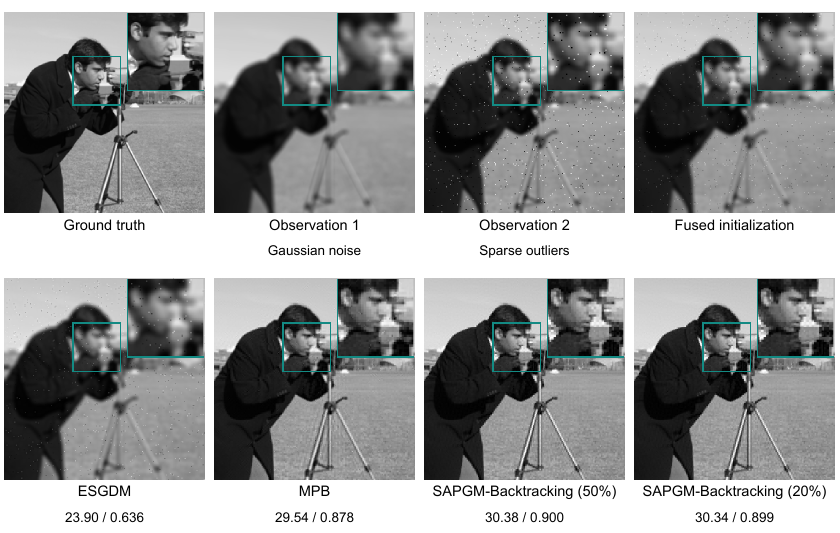}
\caption{Robust Cameraman restoration after 1500 iterations with both curvature reductions. The highest-PSNR start is shown for each configuration. Values are PSNR (dB) / SSIM; insets enlarge the same region.}\label{fig:robust_l1_restoration_images}
\end{figure}
\subsection{Across-image evaluation}
\begingroup
For each restoration model, camera, coins, and moon are combined with two degradation levels and five seeds (2026--2030). The two camera conditions with seed 2026 are used for parameter selection, leaving 28 evaluation conditions. Standard restoration uses 500 iterations and robust restoration uses 1500; robust results are averaged over three starts within each condition. The standard batch has $\lambda_{\rm TV}=0$, so its objective pair is $(f+\lambda_1\|c\|_1,f)$. Figure~\ref{fig:batch-image-summary} reports condition-level quality and running time.

The 50\% and 20\% strategies give mean PSNR/SSIM of 39.083 dB/0.9362 and 38.644 dB/0.9300 in the standard batch, with mean times of 47.08 and 44.51 seconds. In the robust batch, the corresponding values are 33.185 dB/0.9110 and 33.203 dB/0.9120, with mean times of 163.98 and 157.00 seconds. Both strategies outperform the nonsmooth baselines in recorded reconstruction quality and runtime, while scalar FISTA remains better in the standard model. The milder reduction saves time with slightly lower standard-image quality and similar robust-image quality.
\endgroup

\begin{figure}[!htbp]\centering
\includegraphics[width=\textwidth]{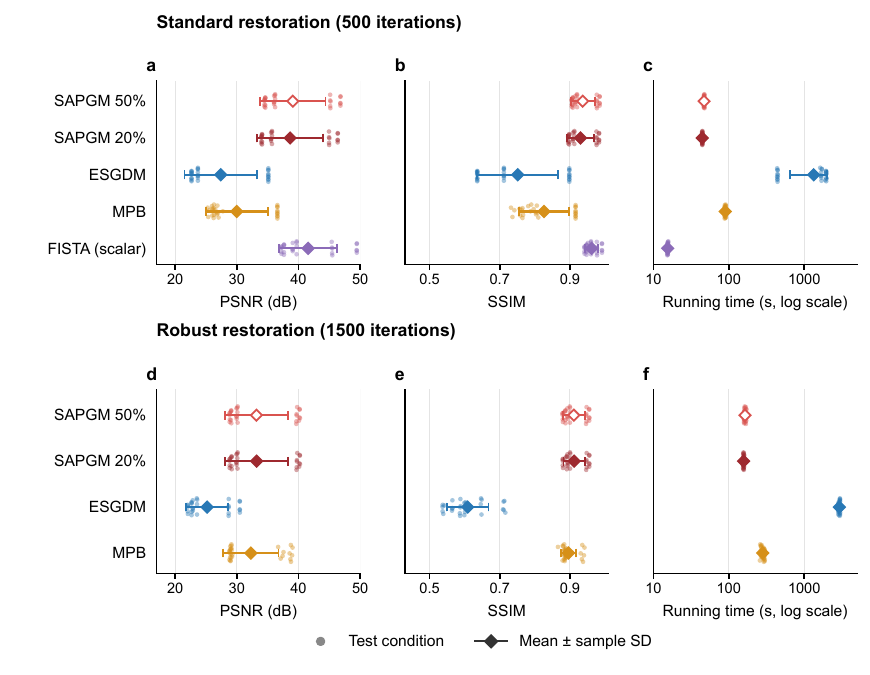}
\caption{Across-image PSNR, SSIM, and running time for standard (top) and robust (bottom) restoration. Points denote 28 test conditions; diamonds and bars show means and sample standard deviations.}
\label{fig:batch-image-summary}
\end{figure}
\FloatBarrier

\section{Conclusions}
\label{sec:conclusions}

We have developed a smoothing accelerated proximal gradient
method with nonmonotone backtracking for convex nonsmooth
composite multiobjective optimization. The method addresses
two computational difficulties associated with smoothing:
growing curvature bounds restrict admissible step sizes,
while approaching a solution of the original problem by
solving successively refined smooth approximation problems
can require substantial computational work. It smooths
the forward terms while retaining the convex proximal
terms in the subproblem. Each backtracking search starts
from a reduced base smoothness estimate, and every trial
recomputes the momentum, extrapolation point, and proximal
candidate using the corresponding effective curvature
$\ell=L/\mu$. The accepted base estimates may increase
or decrease between outer iterations. This construction
adapts both the step size and the extrapolation to the
changing smoothing model.

Assuming that all iterates remain in a compact level set
and that the base smoothness estimates are uniformly
bounded above and away from zero, we have established
convergence rates for a merit function of the original
nonsmooth problem. For
$\mu_k=\mu_0/(k+1)^\sigma$, the rates are
$\mathcal{O}(k^{-\sigma})$,
$\mathcal{O}((\log k)/k)$, and
$\mathcal{O}(k^{\sigma-2})$
for $0<\sigma<1$, $\sigma=1$, and $1<\sigma<2$,
respectively. For $1/2<\sigma<2$, additional summability
conditions on the scalarized-gap and inertial contributions
yield quasi-Fej\'er monotonicity and convergence of the
full iterate sequence to a weak Pareto solution.

Fixed-smoothing experiments show that nonmonotone curvature
adaptation reduces accepted iterations and primitive proximal
calls relative to the tested nondecreasing-backtracking
and APGMO baselines, with terminal objective values matching
at the reported precision. Benchmark and image restoration
results show that the smoothing schedule and the initial
reduction of the base estimate affect both computational
work and solution quality, with the preferred reduction
depending on the test setting. Future work will investigate the strategies for reducing
repeated proximal subproblem solves during backtracking.

\begingroup
\renewcommand{\bibfont}{\small}

\endgroup

\begin{thebibliography}{99}
\setlength{\itemsep}{0pt}
\setlength{\parskip}{0pt}
\bibitem{binder2020multiobjective} Binder M, Moosbauer J, Thomas J, et al. Multi-objective hyperparameter tuning and feature selection using filter ensembles. In: Proceedings of the Genetic and Evolutionary Computation Conference. 2020: 471--479. doi:10.1145/3377930.3389815.

\bibitem{gong2016unsupervised} Gong M, Zhang M, Yuan Y. Unsupervised band selection based on evolutionary multiobjective optimization for hyperspectral images. IEEE Transactions on Geoscience and Remote Sensing, 2016, 54(1): 544--557. doi:10.1109/TGRS.2015.2461653.

\bibitem{combettes2005signal} Combettes P L, Wajs V R. Signal recovery by proximal forward-backward splitting. Multiscale Modeling \& Simulation, 2005, 4(4): 1168--1200.

\bibitem{wu2023smoothing} Wu F, Bian W. Smoothing accelerated proximal gradient method with fast convergence rate for nonsmooth convex optimization beyond differentiability. Journal of Optimization Theory and Applications, 2023, 197(2): 539--572. doi:10.1007/s10957-023-02176-6.

\bibitem{lapucci2023memetic} Lapucci M, Mansueto P, Schoen F. A memetic procedure for global multi-objective optimization. Mathematical Programming Computation, 2023, 15: 227--267. doi:10.1007/s12532-022-00231-3.

\bibitem{gebken2021efficient} Gebken B, Peitz S. An efficient descent method for locally Lipschitz multiobjective optimization problems. Journal of Optimization Theory and Applications, 2021, 188: 696--723. doi:10.1007/s10957-020-01803-w.

\bibitem{hoseini2022bundle} Hoseini Monjezi N, Nobakhtian S. A proximal bundle-based algorithm for nonsmooth constrained multiobjective optimization problems with inexact data. Numerical Algorithms, 2022, 89(2). doi:10.1007/s11075-021-01128-3.

\bibitem{maleknia2024subgradient} Maleknia M, Soleimani-Damaneh M. An effective subgradient algorithm via Mifflin's line search for nonsmooth nonconvex multiobjective optimization. European Journal of Operational Research, 2024, 319(2): 505--516. doi:10.1016/j.ejor.2024.07.019.

\bibitem{sonntag2024hilbert} Sonntag K, Gebken B, M\"{u}ller G, Peitz S, Volkwein S. A descent method for nonsmooth multiobjective optimization in Hilbert spaces. Journal of Optimization Theory and Applications, 2024, 203: 455--487. doi:10.1007/s10957-024-02520-4.

\bibitem{tanabe2019proximal} Tanabe H, Fukuda E H, Yamashita N. Proximal gradient methods for multiobjective optimization and their applications. Computational Optimization and Applications, 2019, 72(2): 339--361.

\bibitem{tanabe2023convergence} Tanabe H, Fukuda E H, Yamashita N. Convergence rates analysis of a multiobjective proximal gradient method. Optimization Letters, 2023, 17(2): 333--350.

\bibitem{tanabe2023accelerated} Tanabe H, Fukuda E H, Yamashita N. An accelerated proximal gradient method for multiobjective optimization. Computational Optimization and Applications, 2023, 86(2): 421--455.

\bibitem{beck2009fast} Beck A, Teboulle M. A fast iterative shrinkage-thresholding algorithm for linear inverse problems. SIAM Journal on Imaging Sciences, 2009, 2(1): 183--202.

\bibitem{nesterov2005smooth} Nesterov Y. Smooth minimization of non-smooth functions. Mathematical Programming, 2005, 103(1): 127--152.

\bibitem{chen2012smoothing} Chen X. Smoothing methods for nonsmooth, nonconvex minimization. Mathematical Programming, 2012, 134: 71--99.

\bibitem{scheinberg2014fast} Scheinberg K, Goldfarb D, Bai X. Fast first-order methods for composite convex optimization with backtracking. Foundations of Computational Mathematics, 2014, 14(3): 389--417.

\bibitem{miettinen1999nonlinear} Miettinen K. Nonlinear Multiobjective Optimization. Kluwer Academic Publishers, 1999.

\bibitem{rockafellar1970convex} Rockafellar R T. Convex Analysis. Princeton University Press, 1970.

\bibitem{parikh2014proximal} Parikh N, Boyd S. Proximal algorithms. Foundations and Trends in Optimization, 2014, 1(3): 127--239.

\bibitem{tanabe2024merit} Tanabe H, Fukuda E H, Yamashita N. New merit functions for multiobjective optimization and their properties. Optimization, 2024, 73(13): 3821--3858.

\bibitem{makela2014mpb} M\"{a}kel\"{a} M M, Karmitsa N, Wilppu O. Multiobjective proximal bundle method for nonsmooth optimization. TUCS Technical Report, 2014.

\bibitem{montonen2017msgdb} Montonen O, Karmitsa N, M\"{a}kel\"{a} M M. Multiple subgradient descent bundle method for convex nonsmooth multiobjective optimization. Optimization, 2018, 67(1): 139--158. doi:10.1080/02331934.2017.1387259.
\bibitem{custodio2011direct} Custodio A L, Madeira J F A, Vaz A I F, Vicente L N. Direct multisearch for multiobjective optimization. SIAM Journal on Optimization, 2011, 21(3): 1109--1140. doi:10.1137/10079731X.

\bibitem{tanabe2022global} Tanabe H, Fukuda E H, Yamashita N. A globally convergent fast iterative shrinkage-thresholding algorithm with a new momentum factor for single and multi-objective convex optimization. arXiv preprint arXiv:2205.05262, 2022.

\bibitem{liuzzi2016derivative} Liuzzi G, Lucidi S, Rinaldi F. A derivative-free approach to constrained multiobjective nonsmooth optimization. SIAM Journal on Optimization, 2016, 26(4): 2744--2774. doi:10.1137/15M1037810.

\bibitem{apolinario2016scalarization} Apolinario H C F, Papa Quiroz E A, Oliveira P R. A scalarization proximal point method for quasiconvex multiobjective minimization. Journal of Global Optimization, 2016, 64(1): 79--96. doi:10.1007/s10898-015-0367-3.

\bibitem{bento2018proximal} Bento G C, Ferreira O P, Sousa Junior J V L. Proximal point method for a special class of nonconvex multiobjective optimization functions. Optimization Letters, 2018, 12: 311--320. doi:10.1007/s11590-017-1114-0.

\bibitem{deb2002fast} Deb K, Pratap A, Agarwal S, Meyarivan T. A fast and elitist multiobjective genetic algorithm: NSGA-II. IEEE Transactions on Evolutionary Computation, 2002, 6(2): 182--197. doi:10.1109/4235.996017.

\bibitem{zhang2007moea} Zhang Q, Li H. MOEA/D: A multiobjective evolutionary algorithm based on decomposition. IEEE Transactions on Evolutionary Computation, 2007, 11(6): 712--731. doi:10.1109/TEVC.2007.892759.

\end{thebibliography}
\end{document}